\documentclass[12pt,reqno]{amsart}

\usepackage[pdftex]{graphicx} 
\usepackage[english]{babel} 
\usepackage[pdftex,linkcolor=black,pdfborder={0 0 0}]{hyperref} 
\usepackage{calc} 
\usepackage{enumitem} 
\usepackage{comment}

\usepackage{import}

\usepackage{adjustbox}

\usepackage{amsmath}
\usepackage{amssymb}
\usepackage{amsthm}
\usepackage[all]{xy}
\usepackage[pdftex]{graphicx}
\usepackage{color}
\usepackage{cite}
\usepackage{url}
\usepackage{indent first}
\usepackage[labelfont=bf,labelsep=period,justification=raggedright]{caption}
\usepackage[english]{babel}
\usepackage[utf8]{inputenc}
\usepackage[colorinlistoftodos]{todonotes}
\usepackage{tkz-fct}
\usepackage{tikz}
\usetikzlibrary{calc}
\usepackage{comment}

\usepackage{mathtools} 
\mathtoolsset{showonlyrefs} 

\usepackage{multicol}
\PassOptionsToPackage{dvipsnames,svgnames}{xcolor}
\usepackage{textcomp}
\usepackage{sistyle}
\SIthousandsep{,}

\usepackage{fullpage}

\usepackage{lipsum} 
\usepackage{import}

\newcommand{\li}{\text{li}}

\newcommand{\K}{\mathbb{K}}

\newcommand{\PP}{\mathfrak{p}}
\newcommand{\QQ}{\mathbb{Q}}

\newcommand{\N}{\mathbb{N}}
\newcommand{\Z}{\mathbb{Z}}

\theoremstyle{plain}
\newtheorem{thm}{Theorem}
\newtheorem{corollary}[thm]{Corollary}

\newtheorem{proposition}[thm]{Proposition}

\newtheorem{lemma}[thm]{Lemma}

\newtheorem{remark}[thm]{Remark}

\newcommand{\half}{\frac{1}{2}}
\newcommand{\invert}[1]{\frac{1}{#1}}
\newcommand{\abs}[1]{\left\lvert#1\right\rvert}

\newcommand{\paranthesis}[1]{\left(#1\right)}

\renewcommand{\Re}{\operatorname{Re}}
\renewcommand{\Im}{\operatorname{Im}}

\begin{document}

\title[Mertens' Third Theorem for Number Fields]{Mertens' Third Theorem for Number Fields: A New Proof, Cram\'er's Inequality, Oscillations, and Bias}

\author[S.~Hathi]{Shehzad Hathi}
	\address{School of Science, UNSW Canberra, Northcott Drive, Australia ACT 2612} 
	\email{shehzadhathi@outlook.com}

\author[E.~S.~Lee]{Ethan S. Lee}
	\address{School of Science, UNSW Canberra, Northcott Drive, Australia ACT 2612} 
	\email{ethan.s.lee@student.adfa.edu.au}

\maketitle

\begin{abstract}
The first result of our article is another proof of Mertens' third theorem in the number field setting, which generalises a method of Hardy. The second result concerns the sign of the error term in Mertens' third theorem. Diamond and Pintz showed that the error term in the classical case changes sign infinitely often and in our article, we establish this result for number fields assuming a reasonable technical condition. In order to do so, we needed to prove Cram\'er's inequality for number fields, which is interesting in its own right.
Lamzouri built upon Diamond and Pintz's work to prove the existence of the logarithmic density of the set of real numbers $x \ge 2$ such that the error term in Mertens' third theorem is positive, so the third result of our article generalises Lamzouri's results for number fields. We also include numerical investigations for the number fields $\mathbb{Q}(\sqrt{5})$ and $\mathbb{Q}(\sqrt{13})$, building upon similar work done by Rubinstein and Sarnak in the classical case.


\end{abstract}

\section{Introduction}

Suppose that a number field $\K$ has degree $n_{\K}$, discriminant $\Delta_{\K}$, and ring of integers $\mathcal{O}_{\K}$. The Dedekind zeta-function associated to $\K$, denoted $\zeta_{\K}(s)$, is regular throughout $\mathbb{C}$ aside from one pole at $s=1$ which is simple and has residue $\kappa_{\K}$. Recall that the Generalised Riemann Hypothesis (GRH) postulates that every non-trivial zero of $\zeta_{\K}(s)$ lies on the line $\Re(s) = 1/2$. Moreover, the assumption that the positive imaginary parts of these zeros are linearly independent over $\mathbb{Q}$ will be referred to as the Generalised Linear Independence Hypothesis (GLI). Throughout this paper, we use the notations $\ll_\K$ and $O_{\K}$, in which the the implied constant may depend on the invariants of $\K$.

\subsection*{Background}

In 1874, Mertens \cite{mertens1874ein} established the product formula
\begin{equation}\label{eqn:MertensMTT}
    \prod_{p\leq x} \left(1 - \frac{1}{p}\right)^{-1} = e^{\gamma}\log{x} + O(1),
\end{equation}
where $\gamma$ is the Euler--Mascheroni constant.
Without an explicit description of the error term, Lebacque \cite{Lebacque} and Rosen \cite{rosen1999generalization} generalised \eqref{eqn:MertensMTT} for number fields $\K$ with $n_{\K}\geq 2$:
\begin{equation}\label{eqn:RosenMTT}
    \prod_{N(\PP)\leq x} \left(1 - \frac{1}{N(\PP)}\right)^{-1} = e^{\gamma}\kappa_{\K}\log{x} + O_{\K}(1),
\end{equation}
where $\kappa_\K$ is the residue of the pole of $\zeta_\K$ at $s = 1$. The product in \eqref{eqn:RosenMTT} runs over the prime ideals $\PP$ of $\mathcal{O}_{\K}$, where $N(\PP)$ denotes the norm of $\PP$. Garcia and the second author \cite{GarciaLeeRamanujan} have established \eqref{eqn:RosenMTT} with an explicit description of the error term for $x\geq 2$.


In the setting $\K = \mathbb{Q}$, Rosser and Schoenfeld \cite{Rosser} observed that
\begin{equation}\label{eqn:RosserSchoenfeldLamzouri}
    \prod_{p\leq x} \left(1 - \frac{1}{p}\right)^{-1} > e^{\gamma}\log{x},
    \quad\text{for}\quad
    2 \leq x \leq 10^8.
\end{equation}
This is an inequality between the product and main term in \eqref{eqn:MertensMTT}. Building upon this observation, Diamond and Pintz \cite{DiamondPintz} have shown that these quantities actually take turns exceeding one other. More precisely, they showed that
\begin{equation*}
    \sqrt{x} \paranthesis{\prod_{p\leq x} \left(1 - \frac{1}{p}\right)^{-1} - e^{\gamma}\log{x}}
\end{equation*}
attains arbitrarily large positive and negative values as $x \to \infty$. Suppose that
\begin{equation}\label{eqn:mathcalMKDef}
    \mathcal{M}_{\K} = \left\{ x\geq 2: \prod_{N(\PP)\leq x}\left(1 - \frac{1}{N(\PP)}\right)^{-1} > e^{\gamma}\kappa_{\K}\log{x} \right\}.
\end{equation}
Lamzouri \cite{Lamzouri} showed $\mathcal{M}_{\mathbb{Q}}$ and its complement have positive lower logarithmic densities under the Riemann Hypothesis (RH). Moreover, assuming RH and the Linear Independence Hypothesis (LI), the set $\mathcal{M}_{\mathbb{Q}}$ has a logarithmic density which is known to be $\delta(\mathcal{M}_{\mathbb{Q}}) = 0.99999973\ldots$ \cite[Theorem~1.3]{Lamzouri}. Therefore, the error term in Mertens' product formula has a strong bias towards the positive sign. Further, it turns out that the logarithmic density of $\mathcal{M}_{\mathbb{Q}}$ is equal to the logarithmic density of the set of real numbers $x \ge 2$ such that $\pi(x) < \text{Li}(x)$ (see proof of Theorem 1.3 in \cite{Lamzouri}). The latter was calculated conditionally (on RH and LI) by Rubinstein and Sarnak in \cite{rubinstein1994chebyshev}.

\subsection*{Results}

In Section \ref{sec:Preliminaries}, we introduce several results which will be important for proving the main results of this paper. One preliminary result we needed to generalise to number fields is a well-known inequality of Cram\'{e}r; see \cite{Cramer} or \cite[Thm.~13.5]{MontgomeryVaughan}.
This generalisation of Cram\'{e}r's bound is presented in Theorem \ref{thm:CramerBound}, and may be of independent interest.

In Section \ref{sec:HardysApproach}, we provide another proof of \eqref{eqn:RosenMTT}, using a different technique to Lebacque \cite{Lebacque} or Rosen \cite{rosen1999generalization}. Our motivation for sharing this new proof of \eqref{eqn:RosenMTT} is that it is not well known and it generalises a method due to Hardy \cite{HardyNote1927}.\footnote{This method of proof was also suggested to the second author and S.~R.~Garcia during another project by T.~Freiberg, who used a similar technique for a different setting in \cite{BardestaniFreiberg}.}

In Section \ref{sec:Oscillations}, we address \cite[Qn.~16]{GarciaLeeRamanujan}, in which the second author and Garcia raised the question whether the difference,
\begin{equation*}
    \Delta = \prod_{N(\PP)\leq x}\left(1 - \frac{1}{N(\PP)}\right)^{-1} - e^{\gamma}\kappa_{\K}\log{x},
\end{equation*}
changes sign infinitely often when $\K \neq \mathbb{Q}$. To this end, we prove Theorem \ref{thm:Main}, which is a number field analogue of \cite[Thm.~1.1]{DiamondPintz}, that demonstrates $\Delta$ \textit{does} change sign infinitely often when $\K \neq \mathbb{Q}$, assuming a reasonable, technical condition.

\begin{thm}\label{thm:Main}
If there exists a non-real zero $\sigma_\K$ of $\zeta_\K(s)$ such that $1/2 \le \Re(\sigma_\K) < 1$ and there is no zero in the right-half plane $\Re(s) > \Re(\sigma_\K)$, then the quantity
\begin{equation*}
    E_1(x) := \sqrt{x}\left\{\prod_{N(\PP)\leq x}\left(1 - \frac{1}{N(\PP)}\right)^{-1} - e^{\gamma}\kappa_{\K}\log{x}\right\}
\end{equation*}
attains arbitrarily large positive and negative values as $x\to\infty$.
\end{thm}

The technical condition in Theorem \ref{thm:Main} has been carefully worded, so that one can apply Landau's oscillation theorem at the opportune moment; it ensures that the right-most zero of $\zeta_\K(s)$ is not a Landau--Siegel (or exceptional) zero.

\begin{remark} \label{rem:exceptionalzeroexists}
In the case where an exceptional zero exists, we expect an additional term in $E_1(x)$ (corresponding to the exceptional zero) to recover a similar type of oscillatory behaviour. However, we have not pursued that line of investigation in this paper.
\end{remark}

In Section \ref{sec:Bias}, we consider the set $\mathcal{M}_{\K}$.  Recall that the \textit{lower} and \textit{upper logarithmic densities} of a set $S \subset [0,\infty)$ are defined respectively as
\begin{equation*}
    \underline{\delta}(S) = \liminf_{x \to \infty} \invert{\log x} \int_{t \in S \cap [2,x]} \frac{dt}{t} \;\; \text{ and } \;\; \overline{\delta}(S) = \limsup_{x \to \infty} \invert{\log x} \int_{t \in S \cap [2,x]} \frac{dt}{t}.
\end{equation*}
When $\underline{\delta}(S) = \overline{\delta}(S) = \delta(S)$, we say that $\delta(S)$ is the \textit{logarithmic density} of $S$. Generalising Lamzouri's work in \cite{Lamzouri}, we show conditionally that $\mathcal{M}_{\K}$ and its complement are unbounded.

\begin{thm}\label{thm:lowerupperdensity}
Assume GRH. Then, for any number field $\K$, $\underline{\delta}(\mathcal{M}_{\K}) > 0$ and $\overline{\delta}(\mathcal{M}_{\K}) < 1$.
\end{thm}

Moreover, assuming GRH and GLI, we calculate the logarithmic density (see Table \ref{table:logdensity}) for two quadratic fields, $\K = \mathbb{Q}(\sqrt{5})$ and $\K = \mathbb{Q}(\sqrt{13})$, adapting the numerical work done by Rubinstein and Sarnak in \cite{rubinstein1994chebyshev} concerning Chebyshev's bias. These computations are useful beyond Mertens' third theorem for number fields because the logarithmic density of $\mathcal{M}_{\K}$ is equal to $\delta(P_\K)$, the logarithmic density of the set of reals $x \ge 2$ such that the error term in the prime ideal theorem is negative, i.e. $\pi_\K(x) - \text{Li}(x) < 0$. We also show that the logarithmic density of $\mathcal{M}_{\K}$ (and consequently $P_\K$) goes to $1/2$ as the discriminant of the quadratic field grows. This phenomenon is referred to as \textit{dissipation of bias}.




\subsection*{Acknowledgements}
We thank Tim Trudgian and Youness Lamzouri for helpful feedback on this project.
We also thank Tristan Freiberg for bringing Hardy's approach in \cite{HardyNote1927} to the second author's attention, and Stephan Garcia for helpful comments and discussions on its implementation. We also thank Greg Martin and Peter Humphries for the helpful correspondence, especially concerning Remark \ref{rem:exceptionalzeroexists}.

\section{Preliminary Results}\label{sec:Preliminaries}

\subsection{The Dedekind zeta-function}\label{ssec:DZF}

Suppose that the degree $n_{\K} = r_1 + 2r_2$, in which $r_1$ is the number of real places and $r_2$ is the number of complex places of $\K$. Further, suppose $r = r_1 + r_2 - 1$, $R_{\K}$ is the regulator of $\K$, and $h_{\K}$ is the class number of $\K$. Landau established all of the knowledge we state here in \cite{LandauEinfuehrung}.

The Dedekind zeta-function is denoted and defined for $\sigma > 1$ by
\begin{equation*}
    \zeta_{\K}(s) = \sum_{\mathfrak{a}} N(\mathfrak{a})^{-s} = \prod_{\PP} \left(1 - N(\PP)^{-s}\right)^{-1},
\end{equation*}
which converges absolutely. Now, $\zeta_{\K}(s)$ may be continued to the entire plane $\mathbb{C}$, apart from a simple pole at $s=1$ using a functional equation. That is, $\zeta_{\K}(s)$ is regular for all $s\in\mathbb{C}$, aside from one simple pole at $s=1$ whose residue is
\begin{equation*}
    \kappa_{\K} = \frac{2^{r_1 + r_2}\pi^{r_2}h_{\K}R_{\K}}{w_{\K}|\Delta_{\K}|^{\frac{1}{2}}};
\end{equation*}
this is called the analytic class number formula.

At $s=0$, $\zeta_{\K}(s) = 0$ as long as $r = r_1 + r_2 - 1 > 0$ and this zero at $s=0$ has order $r$. If $r=0$, then $\K$ is $\mathbb{Q}$ satisfying $(r_1,r_2) = (1,0)$ \textit{or} $\K$ is an imaginary quadratic field satisfying $(r_1,r_2) = (0,1)$.
Moreover, $\zeta_{\K}(s) = 0$ whenever $s$ is a negative, \textit{even} integer (these zeros have order $r_1+r_2$) or $s$ is a negative, \textit{odd} integer (these zeros only occur when $r_2>0$ and they have order $r_2$). Alongside the zero at $s=0$ (whenever $r>0$), these zeros are called \textit{trivial}. The \textit{non-trivial} zeros of $\zeta_{\K}(s)$ satisfy $0 < \Re(s) < 1$, and we note that there might exist a single, simple, real zero $0 < \beta_0 < 1$, which is called the \textit{exceptional zero}. Explicit bounds for $\beta_0$ may be found in \cite{Ahn, Kadiri, Lee}.

\subsection{The prime ideal theorem}

Let $s = \sigma + it$, $\mathfrak{a}$ denote an integral ideal of $\K$, and $\PP$ denote a prime ideal of $\K$.
Suppose that
\begin{equation*}
    \psi_{\K}(x) = \sum_{N(\mathfrak{a})\leq x} \Lambda_{\K}(\mathfrak{a})
    \qquad\text{where}\qquad
    \Lambda_{\K}(\mathfrak{a}) =
    \begin{cases}
    \log{N(\PP)} & \text{if }\mathfrak{a} = \PP^m,\\
    0 & \text{otherwise.}
    \end{cases}
\end{equation*}

The prime ideal theorem was initially proved by Landau in \cite{Landau}. Explicit conditional versions of this theorem have been established in \cite{GrenieMolteniII}, and an explicit, unconditional generalisation has been established by Lagarias and Odlyzko in \cite{LagariasOdlyzko}. Corollary \ref{thm:explicit_formula_final} is a special case of \cite[Thm.~7.1]{LagariasOdlyzko}, and can be obtained using Kadiri and Ng's zero-density estimate from \cite{KadiriNg} with $L=K=\K$ in Lagarias and Odlyzko's notation.\footnote{Greni\'{e} \textit{et al.} make the same claim in \cite[Eqn.~(6)]{GrenieMolteniPPI}.}



\begin{corollary}\label{thm:explicit_formula_final}
Suppose $\K$ is a number field such that $n_{\K} \geq 2$ and $2 \leq T \leq x$. Then
\begin{equation}\label{eqn:explicit_formula_final}
    \psi_{\K}(x) = x - \sum_{|\gamma| \leq T} \frac{x^{\varrho}}{\varrho} + R_{\K}(x,T),
\end{equation}
where $R_{\K}(x,T) \ll \frac{x \log{x}}{T} \left(n_{\K} \log{x} + \log{|\Delta_{\K}|}\right) \ll \frac{x \log^2{x}}{T}$ and $\gamma$ denotes the ordinate of a non-trivial zero $\varrho$ of $\zeta_{\K}(s)$.
\end{corollary}


\subsection{Cram\'er's inequality for number fields}

The next result we require is a generalisation of Cram\'er's inequality for number fields, which we present in Theorem \ref{thm:CramerBound}. A consequence of Theorem \ref{thm:CramerBound} is that $\psi_{\K}(x) = x + O(x^{1/2})$ on average over any interval $[x,2x]$ for $x \ge 2$.

\begin{thm}\label{thm:CramerBound}
Assume GRH. For $x\geq 2$, we have
\begin{equation*}
    \int_x^{2x} \paranthesis{\psi_{\K}(t)-t}^2 dt  \ll x^{2}.
\end{equation*}
\end{thm}

Once Theorem \ref{thm:CramerBound} is established, the Cauchy--Schwarz inequality ensures that
\begin{equation*}
    \paranthesis{\int_1^x \abs{\psi_{\K}(t)-t} dt}^2 \le x \int_1^x \paranthesis{\psi_{\K}(t)-t}^2 dt.
\end{equation*}
Therefore,
\begin{equation}\label{eqn:CramerBoundApplication}
    \int_x^{2x} \abs{\psi_{\K}(t)-t} dt \ll \int_1^{2x} \abs{\psi_{\K}(t)-t} dt \ll x^{3/2},
\end{equation}
which is the form of Cram\'er's inequality used in \cite{DiamondPintz}.

\begin{remark}
The implied constant in Theorem \ref{thm:CramerBound} is of independent interest. In the classical setting, Brent \textit{et al.} \cite{brent2021the} estimate this to be $\le 0.8603$.
\end{remark}

Before we can prove Theorem \ref{thm:CramerBound}, we need two lemmas; the first is given in Lemma \ref{lem:useful_for_later}.

\begin{lemma}\label{lem:useful_for_later}
Suppose that $N_{\K}(T)$ is the number of non-trivial zeros (counted with multiplicity) of $\zeta_{\K}(s)$ up to height $T \geq 1$. We have
\begin{equation*}
    N_{\K}(T+1) - N_{\K}(T) \ll \log(T+1)
\end{equation*}
\end{lemma}

\begin{proof}
Kadiri and Ng \cite{KadiriNg} showed that there exist constants $C_1$, $C_2$, $C_3$ such that
\begin{equation}\label{eqn:z_d_estimate}
    \left| N_{\K}(T) - \frac{T}{\pi}\log\left(|\Delta_{\K}|\left(\frac{T}{2\pi e}\right)^{n_{\K}}\right) \right|
    \leq \underbrace{C_1 \left(\log{|\Delta_{\K}|} + n_{\K}\log{T}\right) + C_2\, n_{\K} + C_3}_{r_{\K}(T)}.
\end{equation}
Trudgian \cite{Trudgian}, and more recently Hasanalizade \textit{et al.} \cite{HasanalizadeShenWong}, provided explicit constants for $C_i$, but these are not necessary for our purposes.
It follows that
\begin{align*}
    N_{\K}(T+1) &- N_{\K}(T)\\
    &< \frac{T}{\pi} \log\left(\left(\frac{T+1}{T}\right)^{n_{\K}}\right) + \frac{1}{\pi}\log\left(\left(\frac{T+1}{2\pi e}\right)^{n_{\K}}\right)
    + 2\,r_{\K}(T+1) \\
    &\ll T \log\left(\frac{T+1}{T}\right)
    + \frac{1}{\pi}\log(T+1)
    + 2\,r_{\K}(T+1)\\
    &\ll \log(T+1)
\end{align*}
for all $T \geq 1$.
\end{proof}

Using Lemma \ref{lem:useful_for_later}, we establish Lemma \ref{lem:sum_is_finite}, which is important in the proof of Theorem \ref{thm:CramerBound}.

\begin{lemma}\label{lem:sum_is_finite}
Suppose $\gamma_1$, $\gamma_2$ are ordinates of zeros $\varrho_1$, $\varrho_2$ of $\zeta_{\K}(s)$. Then the sum
\begin{equation*}
    \sum_{\gamma_1, \gamma_2} \frac{1}{|\gamma_1 \gamma_2| (1 + |\gamma_1 - \gamma_2|)} < \infty.
\end{equation*}
\end{lemma}

\begin{proof}
Since the zeros are symmetric about the real axis, we will focus on the case when $\gamma_1 > 0$. For each such $\gamma_1$, we can split the sum into five parts following an analogous method to \cite[Thm.~13.5]{MontgomeryVaughan}. For the sum over $\gamma_2 < -\gamma_1$, we can use Lemma \ref{lem:useful_for_later}:
\begin{equation*}
    \sum_{\substack{\gamma_2 \\ \gamma_2 < -\gamma_1}} \frac{1}{|\gamma_2| (1 + |\gamma_1 - \gamma_2|)} \ll \sum_{\substack{\gamma_2 \\ \gamma_2 < -\gamma_1}} \invert{{\gamma_2}^2} \ll \frac{\log \gamma_1}{\gamma_1}.
\end{equation*}
Similarly, for $|\gamma_2| \le \half \gamma_1$, the sum is $\ll \frac{(\log \gamma_1)^2}{\gamma_1}$. The sum over those $\gamma_2$ for which $\half \gamma_1 < \gamma_2 < \frac{3}{2} \gamma_1$ is also $\ll \frac{(\log \gamma_1)^2}{\gamma_1}$, and the sum over $\gamma_2 \ge \frac{3}{2} \gamma_1$ is $\ll \frac{\log \gamma_1}{\gamma_1}$. Combining these estimates, we obtain (again using Lemma \ref{lem:useful_for_later})
\begin{equation*}
    \sum_{\gamma_1, \gamma_2} \frac{1}{|\gamma_1 \gamma_2| (1 + |\gamma_1 - \gamma_2|)} \ll \frac{(\log \gamma_1)^2}{{\gamma_1}^2} < \infty,
\end{equation*}
which proves the lemma.
\end{proof}

\begin{proof}[Proof of Theorem \ref{thm:CramerBound}]
Our proof is modelled on \cite[Thm.~13.5]{MontgomeryVaughan}.
For $x\geq 2$, using Corollary \ref{thm:explicit_formula_final} with $T=x$, we have
\begin{align*}
    \int_x^{2x} \abs{\psi_{\K}(t)-t}^2 \,dt
    &\leq \int_x^{2x} \Bigg|\sum_{|\gamma| \leq x} \frac{t^{\varrho}}{\varrho}\Bigg|^2 \,dt
    + 2 \int_x^{2x} \Bigg|\sum_{|\gamma| \leq x} \frac{t^{\varrho}}{\varrho}\Bigg| |R_{\K}(t)| \,dt
    + \int_x^{2x} \abs{R_{\K}(t)}^2 \,dt.
\end{align*}
Now, using Lemma \ref{lem:sum_is_finite}, the main term satisfies
\begin{align*}
    \int_x^{2x} \Bigg|\sum_{|\gamma| \leq x} \frac{t^{\varrho}}{\varrho}\Bigg|^2 \,dt
    = \sum_{\substack{\gamma_1, \gamma_2\\|\gamma_1|\leq x}} \left|\frac{1}{\varrho_1 \overline{\varrho_2}} \left[\frac{t^{2 + i(\gamma_1 + \gamma_2)}}{2 + i(\gamma_1 + \gamma_2)}\right]_{x}^{2x}\right|
    &\ll x^2 \sum_{\gamma_1, \gamma_2} \frac{1}{|\varrho_1 \varrho_2| |2 + i(\gamma_1 + \gamma_2)|}\\
    &\ll x^2.
\end{align*}
Next, using GRH, the middle term satisfies
\begin{align*}
    \int_x^{2x} \Bigg|\sum_{|\gamma| \leq x} \frac{t^{\varrho}}{\varrho}\Bigg| |R_{\K}(t)| \,dt
    \ll \int_x^{2x}  t^{\tfrac{1}{2}} \log^2{t} \Bigg|\sum_{|\gamma| \leq x} \frac{t^{i\gamma}}{\varrho}\Bigg|  \,dt
    \ll \int_x^{2x}  t^{\tfrac{1}{2}} \log^4{t}  \,dt 
    \ll x^2.
\end{align*}
Therein, one can show that the sum over zeros in the second inequality is $\ll \log^2{t}$ using partial summation and applying \eqref{eqn:z_d_estimate}.
Finally, the remainder term satisfies
\begin{align*}
    \int_x^{2x} \abs{R_{\K}(t)}^2 \,dt
    \ll \int_x^{2x} \log^4{t} \,dt
    &\ll x \log^4{x}.
\end{align*}
Summing the above estimates, we obtain Theorem \ref{thm:CramerBound}.
\end{proof}







\begin{remark}
By altering the weight in the integral considered in Theorem \ref{thm:CramerBound}, we can obtain an \textit{exact} formula for the integral.\footnote{The authors thank Peter Humphries for raising this comment.} That is, alter the weight to its natural weighting to consider 
\begin{equation*}
    \int_{x}^{2x} \left(\frac{\psi_{\mathbb{K}}(t) - t}{\sqrt{t}}\right)^2 \, \frac{dt}{t}.
\end{equation*}
Apply Corollary \ref{thm:explicit_formula_final} with $T=x$ and Lemma \ref{lem:useful_for_later} to yield
\begin{equation*}
    \int_x^{2x} \paranthesis{\frac{\psi_{\K}(t)-t}{\sqrt{t}}}^2 \frac{dt}{t} = \int_x^{2x} \paranthesis{\sum_{|\gamma| \leq x} \frac{t^{i\gamma}}{\varrho}}^2 \frac{dt}{t} + O\left( \frac{\log^3{x}}{\sqrt{x}}\right).
\end{equation*}
If the interested reader wanted to compute the integral precisely, then careful examination of this main term is a good starting point. However, to explore this further would transcend the scope of this paper, so the authors propose this as an open problem for the future.

\end{remark}

\subsection{Littlewood's result for number fields}

Suppose that
\begin{align*}
    \Pi(x) = \sum_{\ell \geq 1} \frac{\pi(x^{1/\ell})}{\ell},\quad
    \Pi_{\K}(x) = \sum_{\ell \geq 1} \frac{\pi_{\K}(x^{1/\ell})}{\ell},\quad
    \li^*(x) = \int_1^x \frac{1-t^{-1}}{\log{t}}\,dt,
\end{align*}
$\Delta^*(x) = \Pi(x) - \li^*(x)$ and $\Upsilon^*(x) = \Pi_{\K}(x) - \li^*(x)$.
Littlewood \cite{littlewood1914distribution} proved that
\begin{equation}\label{eq:Littlewood}
    \frac{\Delta^*(x)}{x} = \Omega_{\pm}\left(\frac{\log{\log{\log{x}}}}{\sqrt{x}\log{x}}\right).
\end{equation}
Here, we extend Littlewood's result by generalising the steps in Ingham's presentation in \cite{Ingham} of Littlewood's proof.

\begin{lemma}\label{lem:Littlewood_generalised}
Assume GRH. We have
\begin{equation}\label{eq:Littlewood_generalised}
    \frac{\Upsilon^*(x)}{x} = \Omega_{\pm}\left(\frac{\log{\log{\log{x}}}}{\sqrt{x}\log{x}}\right).
\end{equation}
\end{lemma}

\begin{proof}
Suppose that $R_{\K}(x) = \psi_{\K}(x) - x$, $\Psi_{\K}(x) = \int_2^x \psi_{\K}(t)\,dt$, $\tilde{R}_{\K}(x) = \int_2^x R_{\K}(t)\,dt$, and $\rho = 1/2 + i\gamma$ are the non-trivial zeros of $\zeta_{\K}(s)$. Now, as in the classical case (see  \cite[(12.1)]{MontgomeryVaughan}, for example), we have that
\begin{equation}\label{eqn:explicit_fmla_psi1}
    \Psi_{\K}(x) 
    = \int_2^x \psi_{\K}(t)\,dt
    = \frac{x^2}{2} - \sum_{\rho} \frac{x^{\rho+1}}{\rho (\rho +1)} + O(x) .
\end{equation}
So, if the GRH is true, then \eqref{eqn:explicit_fmla_psi1} implies 
\begin{equation}\label{eqn:explicit_fmla_psi1_appld}
    \Psi_{\K}(x) 
    = \frac{x^2}{2} + O(x^{\frac{3}{2}}) 
    \quad\text{and}\quad
    \tilde{R}_{\K} = O(x^{\frac{3}{2}}) .
\end{equation}
Next, integration by parts tells us
\begin{equation*}
    \Pi_{\K}(x)
    = \sum_{2\leq n\leq x} \frac{\Lambda_{\K}(n)}{\log{n}}
    = \frac{\psi_{\K}(x)}{\log{x}} + \int_2^x \frac{\psi_{\K}(t)}{t(\log{t})^2}\,dt 
\end{equation*}
and 
\begin{equation*}
    \int_2^x \frac{du}{\log{u}}
    = \frac{x}{\log{x}} - \frac{2}{\log{2}} + \int_2^x \frac{u\,du}{u(\log{u})^2}.
\end{equation*}
Therefore, we have
\begin{equation*}
    Q(x) 
    = \Pi_{\K}(x) - \li(x)
    = \frac{R_{\K}(x)}{\log{x}} + \int_2^x \frac{R_{\K}(t)}{t(\log{t})^2}\,dt + O(1).
\end{equation*}
As on \cite[p.~104]{Ingham}, we integrate by parts again and apply \eqref{eqn:explicit_fmla_psi1_appld} to obtain
\begin{align*}
    Q(x) 
    &= \frac{R_{\K}(x)}{\log{x}} + \frac{\tilde{R}_{\K}(x)}{x(\log{x})^2} - \int_2^x \tilde{R}_{\K}(t) \frac{d}{dt}\left(\frac{1}{t(\log{t})^2}\right)\,dt + O(1) \\
    &= \frac{R_{\K}(x)}{\log{x}} + O\left(\frac{\sqrt{x}}{(\log{x})^2}\right).
\end{align*}
It follows that
\begin{align*}
    \frac{Q(x) \log{x}}{\sqrt{x} \log\log\log{x}} 
    &= \frac{R_{\K}(x)}{\sqrt{x} \log\log\log{x}} + O\left(\frac{1}{\log{x} \log\log\log{x}}\right).
\end{align*}
This shows that $\frac{Q(x) \log{x}}{\sqrt{x} \log\log\log{x}}$ and $\frac{R_{\K}(x)}{\sqrt{x} \log\log\log{x}}$ have the same limit superior and limit inferior when $x\to\infty$. Moreover, repeating the proof of \cite[Thm.~34]{Ingham} \textit{mutatis mutandis}, we also know that
\begin{equation*}
    R_{\K}(x) = \Omega_{\pm}\left(\sqrt{x}\log\log\log{x}\right),
\end{equation*}
and hence that 
\begin{equation} \label{Qomegaresult}
    Q(x) = \Omega_{\pm}\left(\frac{\sqrt{x} \log\log\log{x}}{\log{x}}\right).
\end{equation}
Finally, we use the relationship
\begin{equation*}
    \li^*(x) = \li(x) - \log\log{x} + O(1),
\end{equation*}
which was the first observation in the proof of \cite[Prop.~4.1]{DiamondPintz}, to note that
\begin{equation*}
    \frac{\Upsilon^*(x)}{x}
    = \frac{Q(x)}{x} + \frac{\li(x) - \li^*(x)}{x}
    = \frac{Q(x)}{x} + O\left(\frac{\log\log{x}}{x}\right) + O(1).
\end{equation*}
Since the first big-$O$ term is asymptotically smaller than the $\Omega_{\pm}$ result we have for $Q(x)/x$ (from \eqref{Qomegaresult}), the result is a natural conclusion.
\end{proof}

\section{Mertens' Product Formula for Number Fields}\label{sec:HardysApproach}

Hardy gave an elegant proof of \eqref{eqn:MertensMTT} in \cite{HardyNote1927}, which consists of four core ingredients. In this section, we have generalised Hardy's arguments into the number field setting. Note that one can use this method to prove other generalisations of \eqref{eqn:MertensMTT}, as long as the prerequisite ingredients are available.


\subsection{Ingredients}
The simplest ingredients we require are partial summation and Chebyshev-type bounds:
\begin{equation*}
    \pi_{\K}(x) = O\left(\frac{x}{\log{x}}\right).
\end{equation*}
To see that we have Chebyshev-type bounds for $\pi_{\K}(x)$, note that
\begin{equation*}
    \pi_{\mathbb{Q}}(x) \leq \pi_{\K}(x) \leq n_{\K}\pi_{\mathbb{Q}}(x),
\end{equation*}
using the fact that every prime ideal $\PP$ lies over a unique rational prime $p$, and there are at most $n_{\K}$ distinct prime ideals lying over each $p$. Now, \cite[Thm.~4.6]{Apostol} and \cite[Eqn.~(3.6)]{Rosser} establish
\begin{equation*}
    \frac{1}{6} \frac{x}{\log{x}} < \pi_{\mathbb{Q}}(x) < 1.25506 \frac{x}{\log{x}}
    \qquad\text{for}\qquad
    x > 1,
\end{equation*}
hence $\pi_{\K}(x) = O(x/\log{x})$, because
\begin{equation*}
   \frac{1}{6} \frac{x}{\log{x}} < \pi_{\K}(x) < 1.25506\,n_{\K} \frac{x}{\log{x}}.
\end{equation*}
We also require the relationship
\begin{equation}\label{Hardy_A}
    \lim_{\delta\rightarrow 0} \left(\int_a^\infty \frac{e^{-\delta t}}{t}dt - \log{\frac{1}{\delta}} \right) = - \log{a} - \gamma ,
\end{equation}
which Hardy states is ``familiar in the theory of the Gamma function''. Finally, we require a Tauberian theorem \cite[Eqn.~(D)]{HardyNote1927} that states if $f(t) = O(1/(t\log{t}))$ and
\begin{equation}\label{eqn:TauberianTheorem}
    J(\delta) = \int_a^\infty f(t)t^{-\delta}dt \to \ell
    \quad\text{as}\quad
    \delta \to 0,
\end{equation}
then $J(0) = \ell$.

\subsection{Proof}

We are now in position to prove \eqref{eqn:RosenMTT} following Hardy's approach.
Using partial summation, we have
\begin{equation}\label{Hardy(1)NFs}
    \sum_{N(\PP)\leq x}\log(1 - N(\PP)^{-s}) = \pi_{\K}(x) \log(1 - x^{-s}) - s \int_2^x \frac{\pi_{\K}(t)}{t(t^s - 1)}\,dt,
\end{equation}
where $\Re(s) > 1$ and $\mathfrak{p}$ are prime ideals in $\K$.
Let $x\rightarrow\infty$ in \eqref{Hardy(1)NFs} and re-write the integral therein to obtain
\begin{align*}
    \log{\zeta_{\K}(s)} =
    \underbrace{s \int_2^\infty \left\{\pi_{\K}(t) - \frac{t}{\log{t}}\right\} \frac{dt}{t^{1+s}}}_{K_1(s)}
    + \underbrace{s \int_2^\infty \frac{\pi_{\K}(t)}{t^{1+s}(t^s - 1)}dt}_{K_2(s)}
    + \underbrace{s \int_2^\infty \frac{t^{-s}}{\log{t}}dt}_{K_3(s)}.
\end{align*}
Next, we have
\begin{align*}
    \lim_{s \to 1} \left( \log{\zeta_{\K}(s)} - K_3(s) \right) &= \log{\kappa_{\K}} + \log\log{2} + \gamma, && \text{(by \eqref{Hardy_A})}\\
    \lim_{s \to 1} K_2(s) &= K_2(1) && \text{(by uniform convergence)}.
\end{align*}
It follows that
\begin{equation*}
    \lim_{s \to 1} K_1(s) = \log{\kappa_{\K}} + \log\log{2} + \gamma - K_2(1).
\end{equation*}
Therefore, it follows from invoking the Tauberian theorem \eqref{eqn:TauberianTheorem} (which we may do by virtue of the Chebyshev-type observation on $\pi_{\K}(x)$), that
\begin{equation}\label{eqn:HardyLineNFs}
    \int_2^\infty \left\{\pi_{\K}(t) - \frac{t}{\log{t}}\right\} \frac{dt}{t^2}
    = \log{\kappa_{\K}}  + \log\log{2} + \gamma - K_2(1).
\end{equation}
Now, suppose $s = 1$ in \eqref{Hardy(1)NFs}, then
\begin{align*}
    &\sum_{N(\PP)\leq x}\log\left(1 - \frac{1}{N(\PP)}\right)\\
    &\quad= \pi_{\K}(x) \log(1 - x^{-1}) - \int_2^x \left\{\pi_{\K}(t) - \frac{t}{\log{t}}\right\} \frac{dt}{t^2} - \int_2^x \frac{\pi_{\K}(t)dt}{t^{2}(t - 1)} - \int_2^x \frac{dt}{t\log{t}}.\\
    &\hphantom{\quad}= - \log{\kappa_{\K}} - \gamma - \log\log{x} + o(1),
\end{align*}
using \eqref{eqn:HardyLineNFs}. This is equivalent to Mertens' third theorem for number fields \eqref{eqn:RosenMTT}, because $\exp(o(1)) = 1 + o(1)$ and
\begin{equation}\label{eqn:vitamin}
    \prod_{N(\PP)\leq x}\left(1 - \frac{1}{N(\PP)}\right)
    = \exp\left(\sum_{N(\PP)\leq x}\log(1 - N(\PP)^{-1})\right)
    = \frac{e^{- \gamma}}{\kappa_{\K}\log{x}} (1 + o(1)).
\end{equation}


\section{Oscillations in Mertens' Third Theorem}\label{sec:Oscillations}

In this section, we will prove Theorem \ref{thm:Main}.
To ensure that the proof of Theorem \ref{thm:Main} can be approached using analytic techniques, we observe that the problem is equivalent to showing
\begin{equation}\label{eq:starting_point}
    - \sum_{N(\PP)\leq x}\log \left(1 - \frac{1}{N(\PP)}\right) - \log\log{x} - \gamma - \log{\kappa_{\K}} 
    \begin{cases}
    > &\eta/\paranthesis{\sqrt{x}\log{x}}\\
    < &-\eta/\paranthesis{\sqrt{x}\log{x}}
    \end{cases},
\end{equation}
for sequences of $x$ which tend to infinity, and any large $\eta > 0$. We begin by reinterpreting the left-hand side of \eqref{eq:starting_point} in Lemma \ref{lem:reinterpretation}. Let
\begin{equation} \label{defn:A(x)}
    A(x) := \int_1^x \frac{d\Pi_{\K}(t)}{t} - \int_1^x \frac{1-t^{-1}}{t\log{t}} dt - \log \kappa_{\K}.
\end{equation}
Then, we will deduce that
\begin{equation}\label{eq:starting_point_reinterpreted}
    A(x) + O\left(\frac{1}{\sqrt{x}\log{x}}\right)
    \begin{cases}
    > &\eta/\paranthesis{\sqrt{x}\log{x}}\\
    < &-\eta/\paranthesis{\sqrt{x}\log{x}}
    \end{cases},
\end{equation}
is true for sequences of $x$ which tend to infinity, and any large $\eta > 0$.

We consider two separate cases to prove this assertion. First, in Section \ref{ssec:no_GRH}, we assume that the GRH is \textit{not} true and the technical condition in the statement of Theorem \ref{thm:Main} holds, and use Landau's oscillation theorem \cite[Thm.~6.31]{bateman2004analytic}. Second, in Section \ref{ssec:GRH}, we \textit{assume} the GRH, and use Cram\'{e}r's inequality for number fields \eqref{eqn:CramerBoundApplication}. This will automatically prove Theorem \ref{thm:Main}.

\begin{lemma}\label{lem:reinterpretation}
The left-hand side of \eqref{eq:starting_point} is equivalent to
\begin{equation} \label{eq:starting_point_recast}
    \underbrace{\int_1^x \frac{d\Pi_{\K}(t)}{t} - \int_1^x \frac{1-t^{-1}}{t\log{t}} dt - \log \kappa_{\K}}_{A(x)} + O\left(\frac{1}{\sqrt{x}\log{x}}\right).
\end{equation}
\end{lemma}

\subsection{Proof of Lemma \ref{lem:reinterpretation}}

To prove Lemma \ref{lem:reinterpretation}, we first need to convert the sum in \eqref{eq:starting_point} into an integral using Lemma \ref{lem:log_sum_converter}.

\begin{lemma}\label{lem:log_sum_converter}
For $x\geq 2$,
\begin{equation*}
    - \sum_{N(\PP)\leq x}\log\left(1 - \frac{1}{N(\PP)}\right) = \int_1^x \frac{d\Pi_{\K}(t)}{t} + O\left(\frac{1}{\sqrt{x}\log{x}}\right).
\end{equation*}
\end{lemma}

\begin{proof}
Clearly,
\begin{equation*}
    \int_1^x \frac{d\Pi_{\K}(t)}{t}
    = \sum_{\ell \geq 1} \frac{1}{\ell} \int_1^x \frac{d\pi_{\K}(t^{1/\ell})}{t}
    = \sum_{N(\PP)^{\ell}\leq x} \frac{1}{\ell\, N(\PP)^{\ell}}.
\end{equation*}
Moreover, there are at most $n_{\K}$ prime ideals which lie over any rational prime $p$ in a number field $\K$. It follows that
\begin{align*}
    - \sum_{N(\PP)\leq x}\log\left(1 - \frac{1}{N(\PP)}\right) - \int_1^x \frac{d\Pi_{\K}(t)}{t}
    &= \sum_{\substack{N(\PP) \leq x\\N(\PP)^{\ell}> x}} \frac{1}{\ell\, N(\PP)^{\ell}}\\
    &\leq n_{\K}\sum_{\substack{p\leq x\\p^{\ell}> x}} \frac{1}{\ell\, p^{\ell}}
    \ll \frac{1}{\sqrt{x}\log{x}},
\end{align*}
by the contents of the proof of \cite[Lem.~2.1]{DiamondPintz}.
\end{proof}

Second, we import \cite[Lem.~2.2]{DiamondPintz} in Lemma \ref{lem:log_gamma_kappa_converter}, which replaces $\log\log{x} + \gamma$ with integrals.

\begin{lemma}\label{lem:log_gamma_kappa_converter}
For $x > 1$,
\begin{align*}
    \log\log{x} + \gamma
    &= \int_1^x \frac{1-t^{-1}}{t\log{t}} dt - \int_x^\infty \frac{dt}{t^2\log{t}} = \int_1^x \frac{1-t^{-1}}{t\log{t}} dt + O\left(\frac{1}{x\log{x}}\right).
\end{align*}
\end{lemma}

Lemma \ref{lem:reinterpretation} follows by inserting Lemmas \ref{lem:log_sum_converter} and \ref{lem:log_gamma_kappa_converter} into the left-hand side of \eqref{eq:starting_point}.

\subsection{The non-GRH case}\label{ssec:no_GRH}

Following a natural generalisation of the method presented in \cite[Sec.~2]{DiamondPintz}, we can derive the Mellin formula for $A(x)$ for $\Re(s) > 0$, 
\begin{equation*}
    \widehat{A}(s) := \int_1^\infty t^{-s-1}A(t)\,dt
    = \frac{1}{s}\log\frac{s\,\zeta_{\K}(s+1)}{\kappa_\K (s+1)}.
\end{equation*}
We note that $\widehat{A}(s)$ has a removable singularity at $s = 0$ since $\lim \limits_{s\to 0} s \widehat{A}(s) = 0$. Moreover, as in \cite[Sec.~3]{DiamondPintz}, we replace the error term in \eqref{eq:starting_point_recast} by
\begin{equation*}
    B(x) := \frac{1-x^{-1}}{\sqrt{x}\log x}
\end{equation*}
which is asymptotically equivalent to the original error term. The Mellin transform of $B(x)$ will be
\begin{equation*}
    \widehat{B}(s) = \log \frac{s+3/2}{s+1/2}.
\end{equation*}
We want to show that for a fixed $K$, $A(x) + \eta B(x)$ changes sign infinitely often. Suppose GRH is not true, then $\zeta_\K(s)$ has a zero $\sigma_\K$ with $1/2 < \Re(\sigma_{\K}) < 1$. If we assume, as in Theorem \ref{thm:Main}, that there is no zero in the right-half plane $\Re(s) > \Re(\sigma_\K)$ and $\sigma_\K$ is not real, then
\begin{equation*}
    \log \frac{s\,\zeta_{\K}(s+1)}{\kappa_\K(s+1)}
    \quad\text{has a singularity at}\quad
    \sigma^* = \sigma_\K - 1,
\end{equation*}
with $\Re(\sigma^*) > -1/2$ and $\Im(\sigma^*) \neq 0$. As a result, $\widehat{A}+\eta \widehat{B}$ is holomorphic for $\Re(s) > \Re(\sigma^*)$ but not in any half-plane $\Re(s) > \Re(\sigma^*)-\epsilon$ with $\epsilon > 0$. Moreover, since $\sigma^*$ is not real, $\Re(\sigma^*)$ will be a regular point and so, by Landau's oscillation theorem \cite[Theorem~6.31]{bateman2004analytic}, $A(x)+\eta B(x)$ will change sign infinitely often for any fixed value of $\eta$.

\subsection{The GRH case}\label{ssec:GRH}

Using integration by parts and similar arguments to \cite[Sec.~4]{DiamondPintz}, we see that 
\begin{align}
    A(x)
    &= \frac{\Upsilon^*(x)}{x} + \int_1^x \frac{\Upsilon^*(t)}{t^2}\,dt \label{eq:redefine_Ax}\\
    &= \frac{\Upsilon^*(x)}{x} + \int_1^\infty \frac{\Upsilon^*(t)}{t^2}\,dt - \int_x^\infty \frac{\Upsilon^*(t)}{t^2}\,dt \nonumber \\
    &= \frac{\Upsilon^*(x)}{x} - \int_x^\infty \frac{\Upsilon^*(t)}{t^2}\,dt \label{eq:DP4.2analogue},
\end{align}
where \eqref{eq:redefine_Ax} and \eqref{eq:DP4.2analogue} are analogues of equations (4.1) and (4.2) in \cite{DiamondPintz} respectively. 
Using \eqref{eqn:CramerBoundApplication} and dyadic interval estimates, we can also show that
\begin{equation*}
    \int_x^\infty \frac{\Upsilon^*(t)}{t^2}\,dt \ll \frac{1}{\sqrt{x}\log{x}},
\end{equation*}
hence, using \eqref{eq:Littlewood_generalised}, we see that
\begin{equation*}
    A(x) = \Omega_{\pm}\left(\frac{\log{\log{\log{x}}}}{\sqrt{x}\log{x}}\right).
\end{equation*}
This concludes the proof of Theorem \ref{thm:Main}.

\section{Bias in Mertens' Third Theorem}\label{sec:Bias}

For a set $S \subset [0, \infty)$, \textit{upper} and \textit{lower} logarithmic densities of $S$ are defined by
\begin{equation*}
    \overline{\delta}(S) = \limsup_{x\to\infty} \frac{1}{\log{x}} \int_{t\in S\,\cap\,[2,x]} \frac{dt}{t},
    \qquad\text{and}\qquad
    \underline{\delta}(S) = \liminf_{x\to\infty} \frac{1}{\log{x}} \int_{t\in S\,\cap\,[2,x]} \frac{dt}{t}.
\end{equation*}
Moreover, if $\overline{\delta}(S) = \underline{\delta}(S) = \delta(S)$, then $\delta(S)$ is the \textit{logarithmic density} of $S$.
Recall the definition of $\mathcal{M}_{\K}$ from \eqref{eqn:mathcalMKDef}; we know that $x\in \mathcal{M}_{\K}$ if and only if $E_{\K}(x) > 0$, where
\begin{equation} \label{defn:EK}
    E_{\K}(x) = \sqrt{x} \log{x} \left(\log \prod_{N(\PP)\leq x}\left(1 - \frac{1}{N(\PP)}\right)^{-1} - \log{\kappa_{\K}} - \log\log{x} - \gamma \right) .
\end{equation}

The purpose of this section is to demonstrate that $E_{\K}(x)$ has a limiting distribution assuming the GRH. To this end, always assuming the GRH, we establish a useful explicit formula for $E_{\K}(x)$ in Corollary \ref{cor:LamzouriProp2.1Gen} in Section \ref{ssec:xplicitformula}, and we prove Theorem \ref{thm:lowerupperdensity} in Section \ref{ssec:lowerupperdensities}. In Section \ref{ssec:limitingdistribution}, we prove Theorem \ref{thm:limitingdistribution} which uses the aforementioned explicit formula to show that the limiting distribution of $E_{\K}(x)$ exists under the assumption of the GRH, and Theorem \ref{thm:lowerupperdensity} tells us that the logarithmic density, if it exists, must be positive, but not equal to one. In Section \ref{subsec:numericals}, we confirm this by calculating the logarithmic density (assuming GRH and GLI) in two specific cases and outline a general method to perform computations in other similar cases. Finally, in Section \ref{subsec:dissipate}, we address an important question concerning the keenness of the bias as the discriminant of the quadratic field grows.

\subsection{Explicit formula}\label{ssec:xplicitformula}

The following proposition generalises \cite[Prop.~2.1]{Lamzouri}, and is an explicit formula for $E_{\K}(x)$.

\begin{proposition}\label{prop:LamzouriProp2.1Gen}
Suppose $\K$ is a number field and the exceptional zero $\beta_0$ does not exist (see Remark \ref{rem:exceptional_zero_rem}), then, for any $x\geq 2$ and $T\geq 5$, we have
\begin{align*}
    E_{\K}(x) 
    = 1 + \sum_{|\Im{\varrho}| \leq T} \frac{x^{\varrho - \tfrac{1}{2}}}{\varrho - 1} 
    + O\left(\frac{1}{\log{x}} \left( 1 + \sum_{|\Im{\varrho}| < T} \frac{x^{\Re(\varrho) -\tfrac{1}{2}}}{\Im{\varrho}^2} \right) + \frac{\sqrt{x}}{T}\left((\log{x})^2 + \frac{\log^2{T}}{\log{x}}\right)\right).
\end{align*}
Here, $\varrho$ varies over the non-trivial zeros of $\zeta_{\K}(s)$.
\end{proposition}

Once we know Proposition \ref{prop:LamzouriProp2.1Gen}, we can establish the following corollary, which is conditional on the GRH.

\begin{corollary}\label{cor:LamzouriProp2.1Gen}
Assume GRH and let $\half+i\gamma_n$ represent the non-trivial zeros of $\zeta_{\K}(s)$. Then, for any $x\geq 2$ and $T\geq 5$, we have
\begin{equation} \label{eqn:mertenserrorterm}
E_{\K}(x) = 1 + 2\Re \sum_{0 < \gamma_n < T} \frac{x^{i\gamma_n}}{-\half+i\gamma_n} + O \paranthesis{\frac{\sqrt{x}}{T}\left((\log{x})^2 + \frac{\log^2{T}}{\log{x}}\right)}.
\end{equation}
\end{corollary}

\begin{proof}
Using Lemma \ref{lem:useful_for_later} and $\gamma_n \neq 0$, one can see
\begin{equation} \label{sumofordinates}
    \sum_{\abs{\gamma_n}<T} \invert{\gamma_n^2}  \ll 1.
\end{equation}
Now, Corollary \ref{cor:LamzouriProp2.1Gen} follows easily from Proposition \ref{prop:LamzouriProp2.1Gen} and \eqref{sumofordinates}.
\end{proof}

\begin{remark}
In the setting $\K = \mathbb{Q}$, in \cite[Cor.~1]{BrentAccurate}, Brent, Platt, and Trudgian refined Lehman's lemma \cite[Lem.~1]{Lehman}, and used this to establish the sum in \eqref{sumofordinates} over all $\gamma_n > 0$ is $0.02310499\dots$ upto 28 decimal places. To evaluate (or bound with explicit constants) the sum in \eqref{sumofordinates} one could apply similar techniques as Brent \textit{et al.} \cite{BrentAccurate} or Lehman \cite{Lehman}.
\end{remark}

The remainder of this subsection is dedicated to proving Proposition \ref{prop:LamzouriProp2.1Gen}. To do so, we will argue along similar lines to \cite{Lamzouri}, which means we will require the following lemmas.

\begin{lemma}\label{lem:LamLem2.3Gen}
For $x\geq 2$, we have
\begin{equation*}
    - \sum_{N(\PP)\leq x}\log\left(1 - \frac{1}{N(\PP)}\right)
    = \sum_{N(\mathfrak{a})\leq x} \frac{\Lambda_{\K}(\mathfrak{a})}{N(\mathfrak{a})\, \log{N(\mathfrak{a})}} + \frac{1}{\sqrt{x} \log{x}} + O\left(\frac{1}{\sqrt{x} \log^2{x}}\right).
\end{equation*}
\end{lemma}

\begin{proof}
See \cite[Lem.~2.3]{Lamzouri}, the proof generalises naturally.
\end{proof}

\begin{lemma}\label{lem:LamLem2.4Gen}
For $\alpha > 1$, $x \geq 2$, and $T \geq 5$, we have
\begin{align*}
    \sum_{N(\mathfrak{a})\leq x} \frac{\Lambda_{\K}(\mathfrak{a})}{N(\mathfrak{a})^{\alpha}} 
    = &- \frac{\zeta_{\K}'}{\zeta_{\K}}(\alpha) + \frac{x^{1-\alpha}}{1-\alpha} - \frac{x^{\beta_0 - \alpha}}{\alpha - \beta_0} - \sum_{|\Im{\varrho}| \leq T} \frac{x^{\varrho -\alpha}}{\varrho -\alpha} \\
    &+ O\left(x^{-\alpha} \log{x} + \frac{x^{1-\alpha}}{T} \left(4^{\alpha} + \log^2{x} + \frac{\log^2{T}}{\log{x}}\right) + \frac{1}{T} \sum_{N(\mathfrak{a})\geq 1} \frac{\Lambda_{\K}(\mathfrak{a})}{N(\mathfrak{a})^{\alpha + \tfrac{1}{\log{x}}}}\right),
\end{align*}
where $\varrho$ are the non-trivial zeros of $\zeta_{\K}(s)$ and $\beta_{0}$ is the potential real, exceptional zero.
\end{lemma}

\begin{proof}
Almost all aspects of the proof generalise naturally from \cite[Lem.~2.4]{Lamzouri}, but there are some technical considerations one should be careful with, including the potential exceptional zero. Therefore, we present a generalised summary of Lamzouri's method in \cite[Lem.~2.4]{Lamzouri}, with extra details whenever they are required.

We start by noting that since there are $O(\log T)$ non-trivial zeros of $\zeta_\K(s)$ with ordinate in $[T,T+1]$, there exists a $T_0 \in [T,T + 1]$ which has distance $\gg 1/\log{T}$ from the ordinate of the nearest zero of $\zeta_{\K}(s)$. Now, suppose $c = 1/\log{x}$ and consider the contour integral
\begin{equation*}
    \mathfrak{I}_{\K} = \frac{1}{2\pi i} \int_{c-iT_0}^{c+iT_0} -\frac{\zeta_{\K}'}{\zeta_{\K}}(\alpha + s) \frac{x^s}{s}\,ds .
\end{equation*}
First, evaluate $\mathfrak{I}_{\K}$ using Perron's formula. Analogous to \cite[(2.6)]{Lamzouri}, one obtains
\begin{equation*}
    \mathfrak{I}_{\K}
    = \sum_{N(\mathfrak{a})\leq x} \frac{\Lambda_{\K}(\mathfrak{a})}{N(\mathfrak{a})^{\alpha}} + O\left(x^{-\alpha} \log{x} + \frac{1}{T} \left( x^{1-\alpha} \log^2{x} + \sum_{N(\mathfrak{a})\geq 1} \frac{\Lambda_{\K}(\mathfrak{a})}{N(\mathfrak{a})^{\alpha + \tfrac{1}{\log{x}}}}\right)\right).
\end{equation*}
Second, evaluate $\mathfrak{I}_{\K}$ by moving the line of integration to the line $\Re(s) = -U$, where $U > 0$ is large, and invoking Cauchy's residue theorem. That is,
\begin{equation*}
    \mathfrak{I}_{\K} = \frac{1}{2\pi i} \left( \int_{\mathcal{C}} -\frac{\zeta_{\K}'}{\zeta_{\K}}(\alpha + s) \frac{x^s}{s}\,ds - \sum_{i=2}^4 \int_{\mathcal{C}_i} -\frac{\zeta_{\K}'}{\zeta_{\K}}(\alpha + s) \frac{x^s}{s}\,ds\right),
\end{equation*}
where $\mathcal{C} = \mathcal{C}_1 \cup \mathcal{C}_2 \cup \mathcal{C}_3 \cup \mathcal{C}_4$ is a closed contour such that
\begin{align*}
    \mathcal{C}_1 &= [c - iT_0, c + iT_0], &&
    \mathcal{C}_2 = [c + iT_0, -U + iT_0], \\
    \mathcal{C}_3 &= [-U + iT_0, -U - iT_0], &&
    \mathcal{C}_4 = [-U - iT_0, c - iT_0].
\end{align*}
It follows that $U$ should also be chosen such that $U \neq  \alpha + m$ for any $m\in \mathbb{N}$, so that $- U$ does not equal a trivial zero of $\zeta_{\K}(s)$. To estimate the integrals over $\mathcal{C}_i$, Lamzouri's observations generalise naturally; i.e.
\begin{equation*}
    - \frac{1}{2\pi i} \sum_{i=2}^4 \int_{\mathcal{C}_i} -\frac{\zeta_{\K}'}{\zeta_{\K}}(\alpha + s) \frac{x^s}{s}\,ds \ll \frac{x^{1-\alpha}}{T} \left( 4^{\alpha} + \log{x} + \frac{\log^2{T}}{\log{x}}\right) + \frac{1}{T} \sum_{N(\mathfrak{a}) \geq 1} \frac{\Lambda_{\K}(\mathfrak{a})}{N(\mathfrak{a})^{\alpha + c}} .
\end{equation*}
Next, using the properties of $\zeta_{\K}(s)$ we introduced in Section \ref{ssec:DZF}, invoke Cauchy's residue theorem to evaluate the closed contour integral to obtain
\begin{align}
    \frac{1}{2\pi i}\int_{\mathcal{C}} -\frac{\zeta_{\K}'}{\zeta_{\K}}(\alpha + s) \frac{x^s}{s}\,ds \nonumber
    = &- \frac{\zeta_{\K}'}{\zeta_{\K}}(\alpha) + \frac{x^{1-\alpha}}{1-\alpha} - r_{\K} - \frac{x^{\beta_0 - \alpha}}{\alpha - \beta_0} \nonumber
    - \sum_{|\Im{\varrho}| \leq T} \frac{x^{\varrho -\alpha}}{\varrho -\alpha} \\
    &- r_2 \sum_{m \leq \tfrac{U-1-\alpha}{2}} \frac{x^{-2m - 1 -\alpha}}{2m + 1 +\alpha} - (r_1 + r_2) \sum_{m \leq \tfrac{U-\alpha}{2}} \frac{x^{-2m -\alpha}}{2m +\alpha} , \label{eq:contourintegral}
\end{align}
in which $r_{\K} = x^{-\alpha}/\alpha$ if $r > 0$, and $r_{\K} = 0$ otherwise. Now, we have
\begin{align*}
    \sum_{m \leq \tfrac{U-1-\alpha}{2}} \frac{x^{-2m - 1 -\alpha}}{2m + 1 +\alpha} \ll x^{-3 -\alpha},\,
    \sum_{m \leq \tfrac{U-\alpha}{2}} \frac{x^{-2m -\alpha}}{2m +\alpha} \ll x^{-2 -\alpha},\,
    \sum_{T \leq |\Im{\varrho}| \leq T_0} \frac{x^{\varrho -\alpha}}{\varrho -\alpha} \ll \frac{x^{1-\alpha} \log{T}}{T},
\end{align*}
and $r_{\K} \ll x^{-\alpha}$. Substituting these observations in \eqref{eq:contourintegral}, the result follows.
\end{proof}

\begin{lemma}\label{lem:LamLem2.5Gen}
For any $x \geq 2$,
\begin{equation*}
    \log{\zeta_{\K}(\sigma)} + \int_{\sigma}^\infty \frac{x^{1-t}}{1-t}\,dt \to \log{\kappa_{\K}} + \log\log{x} + \gamma,
    \qquad\text{as}\qquad \sigma\to 1^{+}.
\end{equation*}
\end{lemma}

\begin{proof}
See \cite[Lem.~2.5]{Lamzouri}, the proof generalises naturally.
\end{proof}

We are now in a position to prove Proposition \ref{prop:LamzouriProp2.1Gen}.

\begin{proof}[Proof of Proposition \ref{prop:LamzouriProp2.1Gen}]
Suppose $\beta_0$ does \textit{not} exist. Let $\sigma > 1$ be fixed. Using Lemma \ref{lem:LamLem2.4Gen}, we have
\begin{align}
    \sum_{N(\mathfrak{a})\leq x} &\frac{\Lambda_{\K}(\mathfrak{a})}{N(\mathfrak{a})^{\sigma} \log{N(\mathfrak{a})}}
    = \int_{\sigma}^\infty \sum_{N(\mathfrak{a})\leq x} \frac{\Lambda_{\K}(\mathfrak{a})}{N(\mathfrak{a})^{t}}\,dt \nonumber\\
    &= \log{\zeta_{\K}(\sigma)} + \int_{\sigma}^\infty \frac{x^{1-t}}{1-t}\,dt - \sum_{|\Im{\varrho}| \leq T} \int_{\sigma}^\infty \frac{x^{\varrho - t}}{\varrho -t}\,dt + \epsilon(x,T) \label{eqn:into_me}
\end{align}
in which
\begin{align*}
    \epsilon(x,T)
    &\ll \frac{1}{T}\left(\log{x} + \frac{\log^2{T}}{\log^2{x}}\right) + \frac{1}{x} + \frac{1}{T} \sum_{N(\mathfrak{a})\geq 1} \frac{\Lambda_{\K}(\mathfrak{a})}{N(\mathfrak{a})^{1 + \tfrac{1}{\log{x}}} \log{N(\mathfrak{a})}} \\
    &\ll \frac{1}{T}\left(\log{x} + \frac{\log^2{T}}{\log^2{x}}\right) + \frac{1}{x}.
\end{align*}
Therefore, taking the limit as $\sigma\to 1^{+}$ and using Lemma \ref{lem:LamLem2.5Gen}, the right-hand side of \eqref{eqn:into_me} becomes 
\begin{equation} \label{eqn:limitexpression}
    \log{\kappa_{\K}} + \log\log{x} + \gamma - \sum_{|\Im{\varrho}| \leq T} x^{\varrho} \int_1^\infty \frac{x^{- t}}{\varrho -t}\,dt + O\paranthesis{\frac{1}{T}\left(\log{x} + \frac{\log^2{T}}{\log^2{x}}\right) + \frac{1}{x}}.
\end{equation}
Lamzouri \cite[p.~105]{Lamzouri} has shown
\begin{equation*}
    \int_1^\infty \frac{x^{- t}}{\varrho -t}\,dt = \frac{1}{x \log{x} (\varrho - 1)} + O\left(\frac{1}{x \log^2{x} \Im{\varrho}^2}\right).
\end{equation*}
Insert this into \eqref{eqn:limitexpression}, and the result follows using Lemma \ref{lem:LamLem2.3Gen}.
\end{proof}

\subsection{Lower and upper densities}\label{ssec:lowerupperdensities}
Next, we prove Theorem \ref{thm:lowerupperdensity}, which establishes that if the logarithmic density of $\mathcal{M}_{\K}$ exists, then it must be positive but not equal to one. We work along similar lines to Lamzouri's proof of \cite[Thm.~1.1]{Lamzouri}.

\begin{proof}[Proof of Theorem \ref{thm:lowerupperdensity}]
We write $x = e^Y$. Then, we have
\begin{align}
    \invert{\log x} \int_{t \in \mathcal{M}_{\K} \cap [2,x]} \frac{dt}{t} &= \invert{Y} \text{meas}\, \{ \log 2 \le y \le Y \mid e^y \in \mathcal{M}_\K \} \nonumber \\
    & = \invert{Y} \text{meas}\, \{ \log 2 \le y \le Y \mid E_\K (e^y) > 0 \} \label{measureofsetMK}
\end{align}
where $y = \log t$. By Corollary \ref{cor:LamzouriProp2.1Gen} and (\ref{sumofordinates}), we have for $y \ge \log 2$ and $T \ge 5$,
\begin{align*}
    E_{\K}(e^y) &= \sum_{0 < \gamma_n < T} \frac{-\cos(\gamma_n y)+2 \gamma_n \sin(\gamma_n y)}{\invert{4}+\gamma_n^2} + O \paranthesis{1+\frac{e^{y/2}}{T}\paranthesis{y+\frac{\log^2T}{y}}} \\
    &= 2 \sum_{0 < \gamma_n < T} \frac{\sin(\gamma_n y)}{\gamma_n} + O \paranthesis{1+\frac{e^{y/2}}{T}\paranthesis{y+\frac{\log^2T}{y}}}.
\end{align*}
We choose $T = e^Y$ and if $Y$ is large enough, there exists a constant $A > 0$ such that
\begin{equation*}
    2 \paranthesis{\sum_{0 < \gamma_n < e^Y} \frac{\sin(\gamma_n y)}{\gamma_n}-A} < E_{\K}(e^y) < 2 \paranthesis{\sum_{0 < \gamma_n < e^Y} \frac{\sin(\gamma_n y)}{\gamma_n}+A}
\end{equation*}
for all $2 \le y \le Y$. Therefore, from \eqref{measureofsetMK},
\begin{align}
    \invert{\log x} \int_{t \in \mathcal{M}_{\K} \cap [2,x]} \frac{dt}{t} & \ge \invert{Y} \text{meas} \, \left\{ 1 \le y \le Y \mid \sum_{0 < \gamma_n < e^Y} \frac{\sin(\gamma_n y)}{\gamma_n} > A \right\} + O \paranthesis{\invert{Y}}, \label{eqns:lowerdensity} \\
    \invert{\log x} \int_{t \in \mathcal{M}_{\K} \cap [2,x]} \frac{dt}{t} & \le \invert{Y} \text{meas} \, \left\{ 1 \le y \le Y \mid \sum_{0 < \gamma_n < e^Y} \frac{\sin(\gamma_n y)}{\gamma_n} > -A \right\} + O \paranthesis{\invert{Y}}. \label{eqns:upperdensity}
\end{align}
Now, as in \cite[Sec.~2.2]{rubinstein1994chebyshev}, using Littlewood's approach \cite{littlewood1914distribution}, we have
that
\begin{align}
    \invert{Y} \text{meas}\,\left\{ 1 \le y \le Y \mid \sum_{0 < \gamma_n < e^Y} \frac{\sin(\gamma_n y)}{\gamma_n} > \lambda \right\} & \ge c_1 \exp(-\exp(-c_2 \lambda)), \label{eqns:RSmeasurepos} \\
    \invert{Y} \text{meas}\,\left\{ 1 \le y \le Y \mid \sum_{0 < \gamma_n < e^Y} \frac{\sin(\gamma_n y)}{\gamma_n} < -\lambda \right\} & \ge c_1 \exp(-\exp(-c_2 \lambda)) \label{eqns:RSmeasureneg}
\end{align}
for some absolute positive constants $c_1,c_2$, if $Y$ is large enough. Combining \eqref{eqns:lowerdensity}, \eqref{eqns:upperdensity}, \eqref{eqns:RSmeasurepos}, and \eqref{eqns:RSmeasureneg}, we obtain
\begin{equation*}
    \frac{c_1}{2}\exp(-\exp(-c_2 A)) \le \invert{\log x} \int_{t \in \mathcal{M}_{\K} \cap [2,x]} \frac{dt}{t} \le 1 - \frac{c_1}{2} \exp (-\exp(-c_2A)),
\end{equation*}
if $Y = \log x$ is large enough. In particular, $\underline{\delta}(\mathcal{M}_{\K}) > 0$ and $\overline{\delta}(\mathcal{M}_{\K}) < 1$.
\end{proof}

\subsection{Limiting distribution}\label{ssec:limitingdistribution}

Let $\phi:[0,\infty) \to \mathbb{R}$ and let $y_0$ be a non-negative constant such that $\phi$ is square-integrable on $(0,y_0]$. Suppose there exists $(\lambda_n)_{n \in \mathbb{N}}$, a non-decreasing sequence of positive numbers which tends to infinity, $(r_n)_{n \in \mathbb{N}}$, a complex sequence, and $c$ a real constant such that for $y \ge y_0$,
\begin{equation} \label{eqn:phiform}
    \phi(y) =  c + \Re \paranthesis{\sum_{\lambda_n \le X} r_n e^{i\lambda_n y}} + \mathcal{E}(y,X)
\end{equation}
for any $X \ge X_0 > 0$ and $\mathcal{E}(y,X)$ such that
\begin{equation} \label{eqn:phierror}
    \lim_{Y \to \infty} \invert{Y} \int_{y_0}^Y \abs{\mathcal{E}(y,e^Y)}^2 dy = 0.
\end{equation}
The following result, which is a restatement of \cite[Thm.~1.2]{akbary2014limiting}, prescribes the conditions on $(\lambda_n)_{n \in \mathbb{N}}$ and $(r_n)_{n \in \mathbb{N}}$ under which $\phi$ has a limiting distribution.

\begin{thm}\label{thm:akbarylimiting}
Let $\phi:[0,\infty) \to \mathbb{R}$ satisfy \eqref{eqn:phiform} and \eqref{eqn:phierror}. Let $\alpha,\beta > 0$ and $\gamma \ge 0$. Assume either of the following conditions:
\begin{enumerate}
\item $\beta > \half$ and
\begin{equation}
    \sum_{T < \lambda_n \le T+1} \abs{r_n} \ll \frac{(\log T)^\gamma}{T^\beta},
\end{equation}
for $T > 0$.
\item $\beta \le \min\{1,\alpha\}, \alpha^2 + \alpha/2 < \beta^2 + \beta$, and
\begin{equation}
    \sum_{S < \lambda_n \le T} \abs{r_n} \ll \frac{(T-S)^\alpha (\log T)^\gamma}{S^\beta},
\end{equation}
for $T > S > 0$.
\end{enumerate}
Then $\phi(y)$ is a $B^2$-almost periodic function and therefore possesses a limiting distribution. 
\end{thm}
While Theorem \ref{thm:akbarylimiting} can be directly used to show the existence of a limiting distribution, the following corollary (which is a restatement of \cite[Cor.~1.3]{akbary2014limiting}) makes the task much easier.

\begin{corollary}\label{cor:akbarylimiting}
Let $\phi:[0,\infty) \to \mathbb{R}$ satisfy \eqref{eqn:phiform} and \eqref{eqn:phierror}. Assume that $r_n \ll \lambda_n^{-\beta}$ for $\beta > \half$, and
\begin{equation} \label{ineq:rncond}
    \sum_{T < \lambda_n < T+1} 1 \ll \log T.
\end{equation}    
Then $\phi(y)$ is a $B^2$-almost periodic function and therefore possesses a limiting distribution.
\end{corollary}

Using Corollary \ref{cor:akbarylimiting}, we are in a position to state the final result, which establishes that the function $E_{\K}(x)$, defined in \eqref{defn:EK}, has a limiting distribution.

\begin{thm} \label{thm:limitingdistribution}
Suppose GRH is true. Then $E_{\K}(x)$ has a limiting distribution, that is, there exists a probability measure $\mu_\K$ on $\mathbb{R}$ such that
\begin{equation}
    \lim_{x \to \infty} \invert{\log x} \int_2^x f(E_\K(t)) \frac{dt}{t} = \int_{-\infty}^{\infty} f(t) d \mu_\K,
\end{equation}
for all bounded continuous functions $f$ on $\mathbb{R}$.
\end{thm}
\begin{proof}
From \eqref{eqn:mertenserrorterm}, we can see that $E_\K(x)$ can be written in the form given in \eqref{eqn:phiform} by setting
\begin{equation*}
    c = 1,\qquad \lambda_n = \gamma_n,\qquad r_n = \frac{2}{-\half+i\gamma_n},\qquad \text{and}\qquad  y = \log x.
\end{equation*}
To apply Corollary \ref{cor:akbarylimiting}, we must show that $\mathcal{E}(y,X)$ satisfies condition \eqref{eqn:phierror}, that is,
\begin{equation}
    \lim_{Y \to \infty} \invert{Y} \int_0^Y \abs{ \frac{e^{y/2}}{e^Y} \paranthesis{y+\frac{Y^2}{y}} }^2 dy = 0.
\end{equation}
This is straightforward using integration by parts. We also have that $r_n \ll {\gamma_n}^{-1}$, and the number of zeros of $\zeta_\K(s)$ in the interval $(T,T+1)$ is $O(\log T)$ by Lemma \ref{lem:useful_for_later}. This enables us to apply Corollary \ref{cor:akbarylimiting}, whence the result follows.
\end{proof}

Along with GRH, if we also assume GLI, we can prove an explicit formula for the Fourier transform of $\mu_{\K}$.

\begin{proposition}\label{prop:muFT}
Assume GRH and GLI. Then, for any number field $\K$, the Fourier transform of $\mu_{\K}$ is given by
\begin{equation}
    \widehat{\mu_{\K}}(t) = \int_{-\infty}^\infty e^{-it} d\mu_{\K} = e^{-it} \prod_{\gamma_n > 0} J_0 \paranthesis{\frac{2t}{\sqrt{\invert{4}+{\gamma_n}^2}}},
\end{equation}
for all $t \in \mathbb{R}$, where $J_0(t) = \sum_{m=0}^\infty (-1)^m (t/2)^{2m}/{(m!)}^2$ is the Bessel function of order 0.
\end{proposition}
The above proposition is an immediate consequence of Theorem 1.9 of \cite{akbary2014limiting}.

\begin{proposition} \label{prop:rv}
Assume GRH and GLI. Let $X(\gamma_n)$ be a sequence of independent random variables, arranged in increasing order of the positive imaginary parts ($\gamma_n$) of non-trivial zeros of $\zeta_{\K}(s)$, and uniformly distributed on the unit circle. Then $\mu_{\K}$ is the distribution of the random variable
\begin{equation}
    Z = 1 + 2\Re \sum_{\gamma_n > 0} \frac{X(\gamma_n)}{\sqrt{\invert{4}+{\gamma_n}^2}}.
\end{equation}
\end{proposition}
Proposition \ref{prop:rv} generalises naturally from Proposition 4.2 of \cite{Lamzouri} by replacing the ordinates of the zeros of the Riemann zeta-function with the ordinates of the zeros of the Dedekind zeta-function.

Assuming the Riemann Hypothesis (RH) and the Linear Independence Hypothesis (LI) for the Riemann zeta-function, Rubinstein and Sarnak \cite{rubinstein1994chebyshev} showed that the limiting distribution of $(\pi(x)-\text{Li}(x)) (\log x)/\sqrt{x}$ is the distribution of the random variable
\begin{equation}
    \tilde{Z} = -1 + 2 \Re\sum_{\gamma_{\zeta} > 0} \frac{X(\gamma_{\zeta})}{\sqrt{\invert{4}+{\gamma_{\zeta}}^2}}
\end{equation}
where $\gamma_{\zeta}$ are the imaginary parts of the zeros of the Riemann zeta-function. As shown in the proof of Theorem 1.3 of \cite{Lamzouri}, $P(\tilde{Z} > 0)$ is the logarithmic density of the set of real numbers $x \ge 2$ for which $\pi(x) > \text{Li}(x)$. Therefore, $P(Z > 0) = 1-P(\tilde{Z} > 0)$ is the logarithmic density of the set of reals $x \ge 2$ for which $\pi(x) < \text{Li}(x)$. Similarly, from Proposition \ref{prop:rv}, we can also deduce that the logarithmic density of the set $\mathcal{M}_\K$ (assuming GRH and GLI) is the logarithmic density of the set of real numbers $x \ge 2$ for which $\pi_{\K}(x) < \text{Li}(x)$.

\subsection{Numerical investigations}
\label{subsec:numericals}

Let the logarithmic density of the set of real numbers $x \ge 2$ for which $\pi_{\K}(x) < \text{Li}(x)$ be $\delta(P_{\K})$, which is also equal to $\delta(\mathcal{M}_{\K})$, as argued above. For the classical case, one can find the relevant computations in \cite[Sec.~4]{rubinstein1994chebyshev}. Here, we will combine the analysis done for $\delta(P_1^\text{comp})$, $\delta(P_{5;N;R})$, and $\delta(P_{13;N;R})$ in \cite{rubinstein1994chebyshev} to find
\begin{equation*}
    \delta\paranthesis{P_{\mathbb{Q}(\sqrt{5})}}
    \quad\text{and}\quad
    \delta\paranthesis{P_{\mathbb{Q}(\sqrt{13})}} .
\end{equation*}
We can do this due to the following fact about the Dedekind zeta-function of a quadratic number field $\K = \mathbb{Q}(\sqrt{q})$ when $q \equiv 1\mod 4$ is a squarefree integer:
\begin{equation} \label{eqn:dedekindzetadecomp}
    \zeta_{\K}(s) = \zeta(s)L(s,\chi_{1,q}),
\end{equation}
where $\chi_{1,q}$ is the real non-principal character modulo $q$. This is a special case of the factorisation of the Dedekind zeta-function of an abelian number field into a product of Dirichlet $L$-functions (see \cite[Thm.~4.3]{washington1997introduction}, for example). Note that while it is possible to do a similar analysis when $q$ is squarefree and $q \not \equiv 1 \mod 4$, the modulus of the real non-principal character in that case would be $4q$ and we do not have the advantage of utilising the computational work done in \cite{rubinstein1994chebyshev}.

Let $f_{\K}(x)$ be the density function of $\mu_{\K}$. We, instead, consider $\omega_{\K}(x) := f_{\K}(x-1)$ which is symmetric about $x = 0$. Assuming GRH and GLI, from Proposition \ref{prop:muFT}, we know that its Fourier transform is given by
\begin{equation} \label{eqn:fouriertransformofdensity}
    \widehat{\omega_\K}(t) = \prod_{\gamma_n > 0} J_0\paranthesis{\frac{2t}{\sqrt{\invert{4}+{\gamma_n}^2}}},
\end{equation}
where, as in Proposition \ref{prop:rv}, we write the positive ordinates of the zeros of $\zeta_\K$ as $\gamma_n$. In fact, due to \eqref{eqn:dedekindzetadecomp}, the set of zeros with $\gamma_n > 0$ is the union of the sets of zeros with $\gamma_\zeta > 0$ and $\gamma_{\chi_{1,q}} > 0$ . Since \eqref{eqn:fouriertransformofdensity} is analogous to (4.1) in \cite{rubinstein1994chebyshev} and
\begin{equation} \label{eqn:fouriertransformfactorisation}
    \widehat{\omega_\K}(t) = \widehat{\omega_\zeta}(t) \widehat{\omega_{\chi_{1,q}}}(t),
\end{equation}
we will follow the analysis done in \cite{rubinstein1994chebyshev} for $\widehat{\omega_\zeta}(t)$ and $\widehat{\omega_{\chi_{1,q}}}(t)$.

Our objective is to evaluate the integral
\begin{equation} \label{eqn:logdensityintegral}
    \delta(P_\K) = \int_{-\infty}^1 d\omega_\K(t).
\end{equation}
Before continuing with the numerical investigations, we note that $J_0$ is an even function and hence $\widehat{\omega_\K}$ is symmetric about 0. This allows us to easily conclude the following result.
\begin{proposition}
Assume GRH and GLI. Then, for all number fields $\K$,
\begin{equation*}
    \delta(\mathcal{M}_{\K}) = \delta(P_{\K}) > \half.
\end{equation*}
\end{proposition}
Now, as in \cite[(4.2)]{rubinstein1994chebyshev}, \eqref{eqn:logdensityintegral} can be written as
\begin{equation} \label{eqn:logdensityFTintegral}
    \delta(P_\K) = \half + \invert{2 \pi} \int_{-\infty}^\infty \frac{\sin u}{u} \widehat{\omega_\K}(u) du.
\end{equation}
We want to replace the integral in \eqref{eqn:logdensityFTintegral} by a sum that can be evaluated easily. We do this with the help of the Poisson summation formula
\begin{equation} \label{eqn:integraltopoissonsum}
    \invert{2\pi} \int_{-\infty}^\infty \frac{\sin u}{u} \widehat{\omega_\K}(u) du = \epsilon \sum_{n \in \mathbb{Z}}  \varphi(\epsilon n) - \sum_{\substack{n \in \mathbb{Z} \\ n \neq 0}} \widehat{\varphi}\paranthesis{\frac{n}{\epsilon}}.
\end{equation}
where $\epsilon$ is a small number (to be chosen later) and
\begin{align}
    \varphi(u) &= \invert{2\pi} \frac{\sin u}{u} \widehat{\omega_\K}(u),\\
    \widehat{\varphi}(x) &= \half \int_{x-1}^{x+1} d\omega(u).
\end{align}
To estimate the error in replacing the integral in \eqref{eqn:logdensityFTintegral} by the first sum in \eqref{eqn:integraltopoissonsum}, we need a bound on $\widehat{\varphi}(n/\epsilon)$. Following the analysis in \cite{rubinstein1994chebyshev}, it is easy to verify here as well that the magnitude of the error is $< 10^{-20}$ with the choice of $\epsilon$ being $1/20$ for both $\K = \mathbb{Q}(\sqrt{5})$ and $\K = \mathbb{Q}(\sqrt{13})$. Therefore, we have
\begin{equation} \label{eqn:logdensityasinfinitesum}
    \delta(P_\K) = \half + \invert{2 \pi} \sum_{n \in \mathbb{Z}} \epsilon \frac{\sin \epsilon n}{\epsilon n} \widehat{\omega_\K}(\epsilon n) + \text{error}.
\end{equation}
Next, we need to replace the infinite sum $-\infty < n \epsilon < \infty$ with a finite sum $-C \le n\epsilon \le C$ and bound the error in this process. Analogous to \cite[(4.9)]{rubinstein1994chebyshev}, the magnitude of this error is bounded above by
\begin{equation} \label{finitesumerror}
    \frac{\prod_{j=1}^M \paranthesis{\invert{4}+{\gamma_j}^2}^{1/4}}{\pi^{M/2+1}} \paranthesis{\frac{2}{MC^{M/2}}+\invert{20C^{M/2+1}}}
\end{equation}
where $\gamma_j$'s are the ordinates of the zeros of $\zeta_\K$ indexed in increasing order. To generate these zeros, we used Rubinstein's $L$-function calculator in SageMath. For $\K = \mathbb{Q}(\sqrt{5})$, choosing $C = 25$ and $M = 42$, the magnitude of the error is $< 3 \times 10^{-10}$, and for $\K = \mathbb{Q}(\sqrt{13})$, choosing $C = 25$ and $M = 53$, the magnitude of the error is $< 7 \times 10^{-13}$. Therefore, as in \cite[(4.10)]{rubinstein1994chebyshev}, we obtain
\begin{equation} \label{logdensitywithfinitesum}
    \delta(P_\K) = \invert{2 \pi} \sum_{-25 \le n\epsilon \le 25} \epsilon \frac{\sin \epsilon n}{\epsilon n} \prod_{\gamma_n > 0} J_0\paranthesis{\frac{2n\epsilon}{\sqrt{\invert{4}+{\gamma_n}^2}}} + \half + \text{error}.
\end{equation}
Finally, we would like to replace the infinite product in \eqref{logdensitywithfinitesum} with a finite product. In order to do so, we need to introduce a compensating polynomial, $p(t)$ that accounts for the tail of the infinite product:
\begin{equation} \label{eqn:omegaasfiniteproduct}
    \widehat{\omega_\K}(t) = p(t) \prod_{0 < \gamma_n \le X} J_0\paranthesis{\frac{2t}{\sqrt{\invert{4}+{\gamma_n}^2}}} + \text{error}
\end{equation}
for $-C \le t \le C$, where $p(t) = \sum_{m = 0}^A b_m t^{2m}$, and
\begin{equation} \label{eqn:tailofinfiniteproduct}
    \prod_{\gamma_n > X} J_0\paranthesis{\frac{2t}{\sqrt{\invert{4}+{\gamma_n}^2}}} = \sum_{m = 0}^\infty b_m t^{2m}.
\end{equation}
We choose $A = 1$ and $X = 9999$. From the definition of $J_0$ and \eqref{eqn:tailofinfiniteproduct}, we find that $b_0 = 1$ and
\begin{equation}
    b_1 = - \paranthesis{\sum_{\gamma_n > 0} - \sum_{0 < \gamma_n \le X}} \invert{\invert{4}+{\gamma_n}^2}.
\end{equation}
We can evaluate the first sum of $b_1$ using (4.13)--(4.14) and Table 2 of \cite{rubinstein1994chebyshev}. The second sum was computed using Python. This works out to be
\begin{align*}
    b_1 &= -0.000292143\ldots \text{ for } \K = \mathbb{Q}(\sqrt{5})  \text{ and}\\
    b_1 &= -0.000307347\ldots \text{ for } \K = \mathbb{Q}(\sqrt{13}).
\end{align*}

As in \cite{rubinstein1994chebyshev}, the magnitude of the error in \eqref{eqn:omegaasfiniteproduct} is bounded by
\begin{equation} \label{producterrorbound}
    \invert{2\pi} \sum_{-C \le n\epsilon \le C} \epsilon \frac{\abs{\sin n \epsilon}}{\abs{n \epsilon}} \prod_{0 < \gamma_n \le X} \abs{J_0\paranthesis{\frac{2n\epsilon}{\sqrt{\invert{4}+{\gamma_n}^2}}}} \cdot 2 \frac{(T_1 n^2 \epsilon^2)^{A+1}}{(A+1)!},
\end{equation}
where $T_1 = \sum_{\gamma_K > 0} \paranthesis{\invert{4}+{\gamma_n}^2}^{-1}$. For $\K = \mathbb{Q}(\sqrt{5})$, \eqref{producterrorbound} evaluates to $< 7.3 \times 10^{-7}$ and for $\K = \mathbb{Q}(\sqrt{13})$, it evaluates to $< 2 \times 10^{-7}$. Therefore, we finally obtain
\begin{equation}
    \delta(P_\K) = \invert{2\pi} \sum_{-25 \le n\epsilon \le 25} \epsilon \frac{\sin(n\epsilon)}{n\epsilon} (1+b_1 (n\epsilon)^2) \cdot \prod_{0 < \gamma_n \le 9999} J_0\paranthesis{\frac{2n\epsilon}{\sqrt{\invert{4}+{\gamma_n}^2}}} + \half + \text{error},
\end{equation}
with the magnitude of the error being $< 10^{-6}$. We summarise the results in the table below.

\begin{table}[h] 
\centering
\caption{Logarithmic density for $\K = \mathbb{Q}(\sqrt{q})$}
\label{table:logdensity}
\begin{tabular}{|c|c|}
\hline
$q$ & $\delta(P_\K)$ \\
\hline
5 & 0.9876\ldots \\
\hline
13 & 0.9298\ldots \\
\hline
\end{tabular}
\end{table}

\subsection{Dissipation of bias}
\label{subsec:dissipate}

A natural question that arises from the preceding discussion is about the keenness of the bias as $q \to \infty$. For $q = 5$ and $q = 13$, the bias is quite sharp and while it need not be indicative of the overall trend at all, the bias for $q = 13$ is less keen compared to $q = 5$. We will show that, in fact, as $q$ becomes large, the bias completely dissipates and $\delta(P_\K)$ tends to 1/2. For the following discussion, $q = p$ or $2p$, where $p$ is an odd prime. Note that we do not restrict $q$ to be $q \equiv 1 \mod 4$. Indeed, if $q \not \equiv 1 \mod 4$ and $q$ is squarefree, the discriminant of the quadratic field is $4q$ and hence, the $L$-function in \eqref{eqn:dedekindzetadecomp} is $L(s,\chi_{1,4q})$ and so, if $q \to \infty$, $4q \to \infty$ as well. Theorem 1.5 of \cite{rubinstein1994chebyshev} implies that $\delta(P_{q;N;R}) \to \half$ as $q \to \infty$.

We proceed in a similar manner as in \cite[\S~3.2]{rubinstein1994chebyshev}. Recall that Proposition \ref{prop:muFT} gives us
\begin{equation*}
    \widehat{\mu_\K}(t) = e^{-it} \prod_{\gamma_n > 0} J_0\paranthesis{\frac{2t}{\sqrt{\invert{4}+{\gamma_n}^2}}} = e^{it} \cdot \widehat{\mu_\zeta}(t) \cdot \widehat{\mu_{\chi_{1,q}}}(t).
\end{equation*}
We consider $\log \widehat{\mu_\K} (\xi/\sqrt{\log q})$. As in \cite[(3.5)]{rubinstein1994chebyshev}, for a large fixed constant $A$ and $|\xi| \le A$, we can write
\begin{equation} \label{eqn:logmuFTnormal}
    \log \widehat{\mu_\K} \left(\frac{\xi}{\sqrt{\log q}}\right) = -i \frac{\xi}{\sqrt{\log q}} - \frac{\xi^2}{\log q} \sum_{\gamma_n > 0} \invert{\invert{4}+{\gamma_n}^2} + O\paranthesis{\frac{A^4}{\log^2 q} \sum_{\gamma_n > 0} \invert{\paranthesis{\invert{4}+{\gamma_n}^2}^2} },
\end{equation}
where $\gamma_n$ are the ordinates of the zeros of $\zeta_\K$. For our discussion here, $K = \QQ(\sqrt{q})$ and so the set of ordinates will be the union of the set of ordinates of zeros of the Riemann zeta-function and the relevant $L$-function. Since the sum over positive ordinates of zeros of the Riemann zeta-function
\begin{equation*}
    \sum_{\gamma_\zeta > 0} \invert{\invert{4}+{\gamma_\zeta}^2}
\end{equation*}
converges, we can conclude that (as in \cite[\S~3.2]{rubinstein1994chebyshev})
\begin{equation*}
    \sum_{\gamma_n > 0} \invert{\invert{4}+{\gamma_n}^2} = \half \log q + O(\log \log q),
\end{equation*}
taking into account the contribution from the zeros of the corresponding $L$-function. Therefore, \eqref{eqn:logmuFTnormal} becomes
\begin{equation*}
    \log \widehat{\mu_\K} \left(\frac{\xi}{\sqrt{\log q}}\right) = -\half \xi^2 + O \paranthesis{\frac{A}{\sqrt{\log q}} + \frac{A^2 \log\log q}{\log q} + \frac{A^4}{\log q}}.
\end{equation*}
From this, we see that, for $|\xi| \le A$, 
$$\widehat{\mu_\K} \left(\frac{\xi}{\sqrt{\log q}}\right)$$ 
approaches $e^{-\xi^2/2}$ uniformly. We can now use L\'evy's theorem \cite{levy1922determination} to infer that the probability measure $\tilde{\mu}_\K$ corresponding to the normalised limiting distribution 
$$\frac{E_\K(x)}{\sqrt{\log q}}$$ 
converges in measure to the standard Gaussian. Since $\delta(P_\K) = \tilde{\mu}_\K(-\infty,0]$, we can conclude that 
$$\delta(P_\K) = \delta(\mathcal{M}_{\K}) \to \half$$
as $q \to \infty$ ($q = p$ or $2p$).


\bibliographystyle{amsplain}
\bibliography{references}

\end{document}